\documentclass{amsart}
\usepackage{mathrsfs,amsfonts,amsmath,enumerate,booktabs,siunitx,float}

\newtheorem{theorem}{Theorem}
\newtheorem{proposition}{Proposition}[section]

\newtheorem{lemma}[proposition]{Lemma}

\DeclareMathOperator{\tr}{tr}

\DeclareMathOperator{\Int}{Int}

\DeclareMathOperator{\sign}{sign}

\newcommand{\iii}{\mathtt{i}}

\begin{document}
\title[Approximation of the $p$-radius]{Fast approximation of the $p$-radius, matrix pressure or generalised Lyapunov exponent for positive and dominated matrices}

\author{Ian D. Morris}

\maketitle

\begin{abstract}
If $A_1,\ldots,A_N$ are real $d \times d$ matrices then the \emph{$p$-radius}, \emph{generalised Lyapunov exponent} or \emph{matrix pressure} is defined to be the asymptotic exponential growth rate of the sum $\sum_{i_1,\ldots,i_n=1}^N \|A_{i_n}\cdots A_{i_1}\|^p$, where $p$ is a real parameter. Under its various names this quantity has been investigated for its applications to topics including wavelet regularity and refinement equations, fractal geometry and the large deviations theory of random matrix products. In this article we present a new algorithm for computing the $p$-radius under the hypothesis that the matrices are all positive, or more generally under the hypothesis that they satisfy a weaker condition called \emph{domination}. This algorithm is based on interpreting the $p$-radius as the leading eigenvalue of a trace-class operator on a Hilbert space and estimating that eigenvalue via approximations to the Fredholm determinant of the operator. In this respect our method is closely related to the work of Z.-Q. Bai and M. Pollicott on computing the top Lyapunov exponent of a random matrix product. For pairs of positive matrices of low dimension our method yields substantial improvements over existing methods.
\end{abstract}


\section{Introduction} 

If $(A_1,\ldots,A_N)$ is a tuple of real $d \times d$ matrices and $p\in \mathbb{R}$ a real parameter, the limit
\begin{equation}\label{eq:def}\varrho_p(A_1,\ldots,A_N):=\lim_{n \to \infty} \left(\sum_{i_1,\ldots,i_n=1}^N \left\|A_{i_n}\cdots A_{i_1}\right\|^p \right)^{\frac{1}{n}}\end{equation}
exists by applying Fekete's subadditivity lemma to the sequence 
\[a_n(p):=\log \left(\sum_{i_1,\ldots,i_n=1}^N \left\|A_{i_n}\cdots A_{i_1}\right\|^p \right)\]
if $p\geq 0$, or to the sequence $-a_n(p)$ if $p<0$. The quantity \eqref{eq:def}, modulo some trivial variations in its definition, has been studied independently in at least three different contexts and literatures: under the name of \emph{generalised Lyapunov exponent} the quantity $\log(N^{-1}\varrho_p(A_1,\ldots,A_N))$ has been studied for $p \in \mathbb{R}$ in \cite{CrPaVu88,Va10} where its investigation is motivated by the large deviations theory of random matrix products in statistical mechanics; under the name of \emph{matrix pressure}, the quantity $\varrho_p(A_1,\ldots,A_N)$ has been investigated for $p \geq 0$ in the fractal geometry literature in view of its applications to the dimension of self-similar and self-affine limit sets (\cite{FeLa02,FeLoSh19,FeSh14,Mo16,PoVy15}); and in the joint spectral radius literature, the quantity $N^{-1/p}\varrho_p(A_1,\ldots,A_N)^{1/p}$ has been investigated for $p \geq 1$ in connection with its applications to wavelet regularity \cite{CaHeMo04,LaWa95,Wa96} and the control theory of discrete linear inclusions \cite{JuPr10,OgMa13}. Across all three literatures there has arisen the problem of computing or estimating the quantity $\varrho_p(A_1,\ldots,A_N)$ --  as may be seen for example in \cite{JuPr11,Mo16,OgPrJu16,PoVy15,Pr97,StTh18,Va10} -- and it is with this that the present article is concerned. The principal result of this article is a new algorithm for the computation of $\varrho_p(A_1,\ldots,A_N)$ in the case where the matrices $A_1,\ldots,A_N$ are positive and $p$ is an arbitrary real number. More generally, our method extends to the case where the matrices $A_1,\ldots,A_N$ strictly preserve a cone or \emph{multicone}.

\section{Statement of main result}

In order to state our result let us present the definition of a multicone. Let us say that a \emph{cone} in $\mathbb{R}^d$ is a set $\mathcal{K} \subset \mathbb{R}^d$ which is closed, convex, has nonempty interior, satisfies $\lambda\mathcal{K}=\lambda\mathcal{K}$ for all real $\lambda>0$ and satisfies $\mathcal{K}\cap-\mathcal{K}=\{0\}$. A \emph{multicone} will be a tuple $(\mathcal{K}_1,\ldots,\mathcal{K}_m)$ of cones in $\mathbb{R}^d$ such that for some nonzero vector $w \in \mathbb{R}^d$ we have $\langle u,w\rangle>0$ for all nonzero $u \in \bigcup_{j=1}^m \mathcal{K}_j$, and such that $\mathcal{K}_i \cap \mathcal{K}_j=\{0\}$ for distinct $i,j \in \{1,\ldots,m\}$. The vector $w$ is called the \emph{transverse-defining vector} of the multicone. We say that a matrix $A \in M_d(\mathbb{R})$ strictly preserves a cone $\mathcal{K}$ if $A(\mathcal{K}\setminus \{0\}) \subseteq \mathrm{Int} \,\mathcal{K}$, and we say that $A$ strictly preserves a multicone $(\mathcal{K}_1,\ldots,\mathcal{K}_m)$ if for every $i=1,\ldots,m$ we have $A(\mathcal{K}_i\setminus \{0\}) \subseteq \Int \mathcal{K}_j \cup (-\Int \mathcal{K}_j)$ for some $j \in \{1,\ldots,m\}$ depending on $i$. If $A$ strictly preserves a multicone then a simple pigeonhole argument demonstrates that some power of $A$ strictly preserves a cone, which implies that $A$ has a simple leading eigenvalue (which might be either positive or negative). We say that $(A_1,\ldots,A_N)\in M_d(\mathbb{R})^N$ strictly preserves a multicone $(\mathcal{K}_1,\ldots,\mathcal{K}_m)$ if every $A_i$ strictly preserves that multicone. We say that $(A_1,\ldots,A_N)$ is \emph{multipositive} if there exists a multicone which is strictly preserved by $(A_1,\ldots,A_N)$. The property of multipositivity admits characterisations which do not overtly refer to cones or multicones: for example, if $(A_1,\ldots,A_N) \in M_d(\mathbb{R})^N$ is a tuple of invertible matrices then the multipositivity of $(A_1,\ldots,A_N)$ is equivalent to the condition
\begin{equation}\label{eq:dom}\limsup_{n \to \infty} \frac{1}{n}\log \max\left\{\frac{\sigma_2(A_{i_n}\cdots A_{i_1})}{\sigma_1(A_{i_n}\cdots A_{i_1})} \colon 1 \leq i_1,\ldots,i_n \leq N\right\}<0\end{equation}
where $\sigma_k(A)$ denotes the $k^{\mathrm{th}}$ singular value of the matrix $A$, see for example \cite{BaVi12,BoGo09,Mo18}. The condition \eqref{eq:dom} above is sometimes called \emph{$1$-domination} or simply \emph{domination} and has been explored in some detail in the dynamical systems literature \cite{AvBoYo10,BoGo09}; its applications to certain numerical invariants of sets of matrices have been investigated in such works as \cite{BoMo15,BrZe18}.

For each $N \geq 1$ we let $\Sigma_N^*$ denote the set of all finite sequences $\mathtt{i}=(i_1,\ldots,i_n)$ such that $i_1,\ldots,i_n$ are integers between $1$ and $N$. If a tuple of matrices $(A_1,\ldots,A_N)\in M_d(\mathbb{R})^N$ is understood, given $\mathtt{i}=(i_1,\ldots,i_n) \in \Sigma_N^*$ we define $A_\iii := A_{i_n}\cdots A_{i_1}.$
If $\mathtt{i}=(i_1,\ldots,i_n) \in \Sigma_N^*$ then we write $|\iii|:=n$ and call this the length of $\iii$. Finally we let $\rho(A)$ denote the spectral radius of the matrix $A$, and we let $\lambda_1(A),\ldots,\lambda_d(A)$ denote the eigenvalues of $A$ listed in decreasing order of absolute value. Since our matrices $A$ will always strictly preserve a multicone the largest eigenvalue of $A$ will always be unique and the definition of $\lambda_1(A)$ unambiguous.

We may now state the principal result of this article, which is the following:
\begin{theorem}\label{th:main}
Let $(A_1,\ldots,A_N) \in M_d(\mathbb{R})^N$ be multipositive, where $N,d \geq 2$, and let $p \in \mathbb{R}$. For every $n \geq 1$ define
\[t_n:=\sum_{|\iii|=n} \rho(A_\iii)^p \prod_{j=2}^d \left(1-\frac{\lambda_j(A_\iii)}{\lambda_1(A_\iii)}\right)^{-1} = \sum_{|\iii|=n}\frac{\lambda_1(A_\iii)^{d-1}\rho(A_\iii)^p}{p_{A_\iii}'(\lambda_1(A_\iii))}\]
where $p_B(x):=\det (xI-B)$ denotes the characteristic polynomial of the matrix $B$ and $p_B'(x_0)$ its first derivative evaluated at $x_0$. Define $a_0:=1$ and
\begin{align*}a_n&:= \frac{(-1)^n}{n!}\det\begin{pmatrix}
t_1& n -1&  0&\cdots &0 &0\\
t_2&t_1&n-2 &\cdots &0 &0\\
t_3&t_2&t_1 &\ddots &0 &0\\
\vdots & \vdots & \vdots & \ddots  &\ddots& \vdots\\
t_{n-1} &t_{n-2}&t_{n-3}&\cdots &t_1 &1\\
t_n &t_{n-1}&t_{n-2}&\cdots &t_2 &t_1
\end{pmatrix}\\
&=\sum_{k=1}^n \frac{(-1)^k}{k!}\sum_{\substack{n_1,\ldots,n_k \geq 1\\n_1+\cdots +n_k=n}} \prod_{\ell=1}^k \frac{t_{n_\ell}}{n_\ell}\end{align*}
for every $n\geq 1$. Then for all sufficiently large $n$ there exists a smallest positive real root $r_n>0$ of the polynomial $\sum_{k=0}^n a_kx^k$, and there exist constants $K,\gamma>0$ such that for all large enough integers $n$
\begin{equation}\label{eq:order}\left|\varrho_p(A_1,\ldots,A_N)-\frac{1}{r_n}\right| \leq K\exp\left(-\gamma n^{\frac{d}{d-1}}\right).\end{equation}
\end{theorem}
Theorem \ref{th:main} applies in particular if the matrices $A_i$ are all positive matrices, or if the matrices $A_i$ all strictly preserve a single cone $\mathcal{K}$. However, multipositive matrix tuples with neither of these two properties also exist: see \cite{AvBoYo10}. We remark that since $\varrho_p(A_1,\ldots,A_N)=\varrho_p(X^{-1}A_1X,\ldots,X^{-1}A_NX)$ for every invertible matrix $X$, a sufficient condition for the application of Theorem \ref{th:main} is that the matrices $A_i$ be \emph{simultaneously conjugate} to positive matrices.

The reader will notice that the order of convergence in Theorem \ref{th:main} is strongest when the dimension of the matrix is $2$ and becomes weaker as the dimension is increased, although it is in all cases super-exponential in $n$. The problem of estimating the implied constants $K$ and $\gamma$ in \eqref{eq:order} is not attempted in this article; we believe that in the case of tuples of positive matrices this should be feasible in principle, but would rely on difficult functional-analytic estimates such as an \emph{a priori} bound for the cardinality of the relative covers arising in the application of \cite[Theorem 4.7]{BaJe08} to certain complex cones. In any event, convergence in Theorem \ref{th:main} is fast enough to yield significant results in low dimensions. In the previous work \cite{JuPr10}, R. Jungers and V. Yu. Protasov investigated the problem of computing what in our notation corresponds to the quantity
\[1-\log_2 \left(2^{-\frac{1}{p}}\varrho_p(A_1,A_2)^{\frac{1}{p}}\right) = \frac{p+1}{p} -\frac{1}{p}\log_2\varrho_p(A_1,A_2)\]
 for the pair of matrices
\[A_1:=\begin{pmatrix}\frac{1}{5}&0\\ \frac{1}{5}&\frac{3}{5}\end{pmatrix},\qquad A_2:=\begin{pmatrix}\frac{3}{5}&\frac{1}{5}\\ 0&\frac{1}{5}\end{pmatrix} \]
with $p:=3.5$, obtaining an estimate of $1.95 \leq\frac{p+1}{p} -\frac{1}{p}\log_2\varrho_p(A_1,A_2) \leq 1.973$. It happens that the pair $(A_1,A_2)$ is simultaneously conjugate to a pair of positive matrices; taking $n:=20$ in Theorem \ref{th:main} yields the estimate
\begin{equation}\label{eq:showoff}1.953821293179325866750389914731492551138280064126997\ldots\end{equation}
for the same quantity, which is empirically accurate to all decimal places shown.

We remark that in the literature on the generalised Lyapunov exponent, it is common to consider the quantity
\begin{equation}\label{eq:notdef}\lim_{n \to \infty} \left(\sum_{i_1,\ldots,i_n=1}^N p_{i_n}\cdots p_{i_1}\left\|A_{i_n}\cdots A_{i_1}\right\|^p \right)^{\frac{1}{n}}\end{equation}
in place of the quantity $\varrho_p(A_1,\ldots,A_N)$ as defined in \eqref{eq:def}, where $(p_1,\ldots,p_N)$ is a probability vector. The quantity \eqref{eq:notdef} can easily be included within the scope of \eqref{eq:def} and Theorem \ref{th:main} by replacing each instance of a matrix $A_i$ with the corresponding matrix $p_i^{1/p}A_i$. Concretely, this implies that the quantity \eqref{eq:notdef} can be calculated using Theorem \ref{th:main} by taking instead
\[t_n:=\sum_{|\iii|=n}  p_\iii \rho(A_\iii)^p \prod_{j=2}^d \left(1-\frac{\lambda_j(A_\iii)}{\lambda_1(A_\iii)}\right)^{-1} =\sum_{|\iii|=n}p_\iii \frac{\lambda_1(A_\iii)^{d-1}\rho(A_\iii)^p}{p_{A_\iii}'(\lambda_1(A_\iii))}\]
where $p_\iii:=p_{i_1}\cdots p_{i_n}$, and leaving the rest of the theorem unchanged. For the remainder of the article we therefore ignore the issue of giving a probability weighting to each $A_i$ and concentrate on the calculation of the $p$-radius as defined in \eqref{eq:def}.

It is possible to show that the quantities $t_n$ defined in Theorem \ref{th:main} satisfy $\lim_{n \to \infty} t_n \varrho_p(A_1,\ldots,A_N)^{-n}=1$ and therefore increase (or decrease) exponentially with $n$. The efficiency of the estimate in Theorem \ref{th:main} on the other hand relies on the quantities $a_n$ decreasing as $O(\exp(-\gamma n^{\frac{d}{d-1}}))$. The small size of the quantities $a_n$ thus arises from additive cancellation among the relatively large terms in the sum defining each $a_n$. In practical applications it is therefore important to compute the quantities $t_n$ to a precision exceeding that desired for the approximation to $\varrho_p(A_1,\ldots,A_N)$. 

The remainder of this article is structured as follows. In \S\ref{se:review} below we review the fundamental properties of $\varrho_p$ and describe some existing techniques for its estimation. In \S\ref{se:outline} we describe in outline the techniques underlying the proof of Theorem \ref{th:main} and in \S\ref{se:proof} the proof itself is presented. In \S\ref{se:examples} we present some examples of the computation of $\varrho_p$ using the algorithms described herein.


\section{Methods for estimating the $p$-radius}\label{se:review}\subsection{Fundamental estimates}\label{ss:basic}
If $(A_1,\ldots,A_N) \in M_d(\mathbb{R})^N$ and $p \in \mathbb{R}$ then by elementary estimates it follows that $\varrho_p(A_1,\ldots,A_N)=0$ if and only if the joint spectral radius
\[\lim_{n \to \infty}\max_{|\iii|=n} \|A_\iii\|^{\frac{1}{n}}\]
is zero. It is well known that the joint spectral radius is zero if and only if all of the products $A_{i_d}\cdots A_{i_1}$ of length $d$ are zero, if and only if there exists a basis in which all of the matrices $A_1,\ldots,A_N$ are simultaneously upper triangular with all diagonal entries equal to zero (for details see \cite[\S2.3.1]{Ju09}). Since the theory of the $p$-radius is trivial in this situation we will for the remainder of this paper deal only with matrices for which the $p$-radius is assumed to be nonzero. We remark that in the multipositive case considered in Theorem \ref{th:main} every product $A_\iii$ has a simple leading eigenvalue and in particular is not the zero matrix, so in this case $\varrho_p(A_1,\ldots,A_N)$ is guaranteed to be nonzero.

When $p>0$ the $p$-radius admits an elementary description as the limit of a convergent sequence of upper bounds,
\begin{equation}\label{eq:element-up}\varrho_p(A_1,\ldots,A_N)=\lim_{n \to \infty} \left(\sum_{|\iii|=n}  \left\|A_\iii\right\|^p \right)^{\frac{1}{n}}=\inf_{n \geq 1} \left( \sum_{|\iii|=n}\left\|A_\iii\right\|^p \right)^{\frac{1}{n}},\end{equation}
as a consequence of the submultiplicativity relation
\[\sum_{|\iii|=m+n} \|A_\iii\|^p \leq \left(\sum_{|\iii|=m} \|A_\iii\|^p\right)\left(\sum_{|\iii|=n}\|A_\iii\|^p\right).\]
Less trivially, when $p>0$ it may also be expressed as the limit of a convergent sequence of lower bounds:
\begin{align}\label{eq:element-down}\varrho_p(A_1,\ldots,A_N)&=\lim_{n \to \infty} \left(\frac{\sum_{|\iii|=nd} \|A_\iii\|^p}{K(p,d)\left(\sum_{|\iii|=n} \|A_\iii\|^p\right)^{d-1}} \right)^{\frac{1}{n}}\\\nonumber
&=\sup_{n \geq 1} \left(\frac{\sum_{|\iii|=nd} \|A_\iii\|^p}{K(p,d)\left(\sum_{|\iii|=n} \|A_\iii\|^p\right)^{d-1}} \right)^{\frac{1}{n}}\end{align}
where $K(p,d):=d^{2+(d+1)p}\max\{d^{1-p},1\}$, see \cite[Theorem 1.2]{Mo16}. In particular the $p$-radius can in principle be approximated to within any prescribed error $\varepsilon$ by systematically computing the upper and lower bounds until they eventually agree to within the prescribed amount. However, since the computational effort involved increases exponentially with $n$ and the relative error may reasonably be presumed to be at least of the order of $K(p,d)^{1/n}$, and since the constant $K(p,d)$ is relatively large even in the case $d=2$, this procedure seems unlikely to have any value for practical computations. An illustration of this is presented in \S\ref{se:examples} below. We remark that an additional theoretical consequence of the above expressions is that the $p$-radius varies continuously both in $p$ and in the matrix entries when $p$ is positive, since it is then equal to both an upper and a lower pointwise limit of sequences of continuous functions, hence continuous.  When $p<0$ the computability and continuity of the $p$-radius do not seem to have been as thoroughly investigated, but based on the related works \cite{BoMo15,Mo17,TsBl97} it seems likely that continuity should not hold and that systematic upper and lower estimation might be infeasible, at least when the matrices are not assumed to be positive or invertible.

When $p$ is a positive even integer, or when $p$ is a positive integer and the matrices $A_1,\ldots,A_N$ preserve a cone, the identity
\begin{equation}\label{eq:repmet}\varrho_p(A_1,\ldots,A_N)=\rho\left(\sum_{i=1}^N A_i^{\otimes p}\right)\end{equation}
has been discovered independently on several occasions \cite{CrPaVu93,Pr97,Zh98}. (Here $A^{\otimes p}$ denotes the $p^{\mathrm{th}}$ Kronecker power of the matrix $A$, see for example \cite[\S4.2]{HoJo94}.) When $p$ is a positive integer and $A_1,\ldots,A_N$ are not necessarily positive, the inequality
\[\varrho_p(A_1,\ldots,A_N)\leq \rho\left(\sum_{i=1}^N A_i^{\otimes p}\right)\]
may be obtained by the same means. Whilst in principle \eqref{eq:repmet} represents an easy method for computing the $p$-radius of positive matrices, the size of the auxiliary matrix $\sum_{i=1}^N A_i^{\otimes p}$ increases exponentially with $p$ which prevents the use of the formula when $p$ is sufficiently large. For non-integer $p$ these results may nonetheless be exploited so as to yield upper bounds as follows. We observe that if $p_1$ and $p_2$ are real numbers such that $0<p_1<p_2$, and $\lambda \in (0,1)$, then for each $n \geq 1$
\[\sum_{|\iii|=n} \|A_\iii\|^{\lambda p_1 + (1-\lambda)p_2} \leq \left(\sum_{|\iii|=n} \|A_\iii\|^{p_1}\right)^\lambda \left(\sum_{|\iii|=n} \|A_\iii\|^{p_2}\right)^{1-\lambda}\]
using H\"older's inequality with $p:=\frac{1}{\lambda}$ and $q:=\frac{1}{1-\lambda}$. It follows easily that
\[\log \varrho_{\lambda p_1 + (1-\lambda)p_2}(A_1,\ldots,A_N) \leq \lambda\log \varrho_{p_1}(A_1,\ldots,A_N)+ (1-\lambda)\log \varrho_{p_2}(A_1,\ldots,A_N) \]
and hence the function $p \mapsto \log \varrho_p(A_1,\ldots,A_N)$ is convex. This yields the upper bound
\begin{align}\label{eq:logconvex}
\varrho_p(A_1,\ldots,A_N) &\leq \varrho_{\lfloor p\rfloor}(A_1,\ldots,A_N)^{p-\lfloor p\rfloor}\varrho_{1+\lfloor p\rfloor}(A_1,\ldots,A_N)^{1+\lfloor p\rfloor-p} \\ \nonumber
&\leq \rho\left(\sum_{i=1}^N A_i^{\otimes \lfloor p\rfloor}\right)^{p-\lfloor p\rfloor} \rho\left(\sum_{i=1}^N A_i^{\otimes (1+ \lfloor p \rfloor)}\right)^{1+\lfloor p \rfloor-p}\end{align}
valid for all $p>0$ and $A_1,\ldots,A_N \in M_d(\mathbb{R})$, which does not seem to have been previously noted in the literature. We will see in \S\ref{se:examples} below that despite its crudity this estimate does not automatically provide a bad approximation and should not be discounted out of hand.

\subsection{Resampled Monte Carlo methods}\label{ss:van}

In \cite{Va10}, J. Vanneste introduced a method based on the interpretation of the $p$-radius as an asymptotic moment of a random matrix product: given $A_1,\ldots,A_N \in M_d(\mathbb{R})$, $n \geq 1$ and $p \in \mathbb{R}$ we may view the sum $\frac{1}{N^n} \sum_{|\iii|=n} \|A_\iii\|^p$ as the expectation of the random variable $\iii \mapsto \|A_\iii\|^p$ where each word $\iii$ of length $n$ is chosen with probability $1/N^n$. This suggests the possibility of approximating $\frac{1}{N^n} \sum_{|\iii|=n} \|A_\iii\|^p$ for large $n$ by Monte Carlo estimation: if we choose $M$ words $\iii_1,\ldots,\iii_M$ independently then by the law of large numbers, the average $\frac{1}{M}\sum_{k=1}^M \|A_{\iii_k}\|^p$ should for large enough $M$ give a reasonable approximation to the value $\frac{1}{N^n} \sum_{|\iii|=n} \|A_\iii\|^p$ which is that random variable's expectation and hence a good approximation to $\varrho_p(A_1,\ldots,A_N)$ as long as $n$ is reasonably large. However, except which $p$ is small, the variance of this random variable will be prohibitively large -- indeed exponentially large in $n$ -- which makes convergence in the strong law of large numbers unreasonably slow. To compensate for this Vanneste introduced a ``go-with-the-winners'' resampling scheme along the lines of \cite{Gr02}, which successively modifies the distribution of the random variable $\iii \mapsto \|A_\iii\|^p$ so as to retain the same mean while reducing the variance; see discussion in \cite[{\S}III]{Va10} for details.  The particular strength of this method is that it has very limited dependence on the number of matrices and their dimension; on the other hand, the accuracy of the results is relatively low in practice. See \S\ref{se:examples} below for further discussion.

\subsection{The convex optimisation bounds of Jungers and Protasov}\label{ss:jupr}

The article \cite{JuPr11} introduced new systematic upper and lower bounds for the $p$-radius in the case $p \geq 1$. If $(A_1,\ldots,A_N)$ are non-negative matrices, Jungers and Protasov showed that the quantities 
\[\mathsf{a}_p(n)=\inf_{(u_1,\ldots,u_d)\in \mathbb{R}^d} \sum_{|\iii|=n}\left( \max_{1 \leq i \leq d} \sum_{j=1}^d (A_\iii)_{ij}e^{u_j-u_i}\right)^p,\]
\[\mathsf{b}_p(n)=\inf_{(v_1,\ldots,v_d)\in \mathbb{R}^d} \max_{1 \leq j \leq d}\sum_{|\iii|=n}\left( \sum_{i=1}^d (A_\iii)_{ij}e^{v_i-v_j}\right)^p, \]
where $(B)_{ij}$ denotes the $(i,j)$ entry of the matrix $B \in M_d(\mathbb{R})$, satisfy
\[\max\left\{d^{-\frac{p}{n}}\mathsf{a}_p(n)^{\frac{1}{n}},d^{\frac{1-p}{n}}\mathsf{b}_p(n)^{\frac{1}{n}}\right\} \leq \varrho_p(A_1,\ldots,A_N) \leq \mathsf{b}_p(n)^{\frac{1}{n}}\]
for every $n \geq 1$. (Here we have modified the statement of their results in concordance with our definition of $\varrho_p$.) The quantities $\mathsf{a}_p(n)$ and $\mathsf{b}_p(n)$ are solutions to convex optimisation problems and as such may be efficiently approximated. In the case where $(A_1,\ldots,A_N)$ preserves a more general cone $\mathcal{K}$ (in the weak sense that $A_i \mathcal{K} \subseteq \mathcal{K}$ for each $i=1,\ldots,N$) analogous upper and lower bounds are given, but these are not in general the solutions to convex optimisation problems and as such are more difficult to efficiently or rigorously estimate. Since the matrices $A_1^{\otimes 2},\ldots,A_N^{\otimes 2}$ always preserve a cone irrespective of the structure of the original matrices $A_1,\ldots,A_N$, and since $\varrho_p(A_1,\ldots,A_N)=\varrho_{p/2}(A_1^{\otimes 2},\ldots,A_N^{\otimes 2})$ for all $p \in \mathbb{R}$, this more general version of their method permits the estimation of $\varrho_p(A_1,\ldots,A_N)$ for arbitrary $A_1,\ldots,A_N \in M_d(\mathbb{R})$ and $p \geq 2$. 

As with the upper and lower bounds \eqref{eq:element-up} and \eqref{eq:element-down} this system of estimation requires the computation of $N^n$ matrix products in order to obtain the $n^{\mathrm{th}}$ approximation and as such is best suited to cases in which $N$ is small.

\subsection{Eigenvalue methods}\label{ss:eig}

As has been previously observed by J. Vanneste \cite[{\S}II.B]{Va10}, the quantity $\varrho_p(A_1,\ldots,A_N)$ can be represented as the leading eigenvalue of a linear operator on an infinite-dimensional function space in the following manner. 
Suppose that $A_1,\ldots,A_N \in M_d(\mathbb{R})$ are invertible matrices and let $p \in \mathbb{R}$. Let $\mathbb{RP}^{d-1}$ denote the space of lines through the origin in $\mathbb{R}^d$, with the distance between two lines defined to be the angle at which they intersect. For each nonzero $u \in \mathbb{R}^d$ let $\overline{u} \in \mathbb{RP}^{d-1}$ denote the line spanned by $u$. Define an operator on the space $C^\alpha(\mathbb{RP}^{d-1})$ of $\alpha$-H\"older continuous functions $f \colon \mathbb{RP}^{d-1} \to \mathbb{R}$ by
\[\left(\mathcal{L}_pf\right)(\overline{u}):=\sum_{i=1}^N \left(\frac{\|A_iu\|}{\|u\|}\right)^p f(\overline{A_iu})\]
and observe that by a simple calculation
\[\left(\mathcal{L}_p^nf\right)(\overline{u})=\sum_{|\iii|=n} \left(\frac{\|A_\iii u\|}{\|u\|}\right)^p f(\overline{A_\iii u})\]
for every $n \geq 1$, $f \in C^\alpha(\mathbb{RP}^{d-1})$ and $\overline{u} \in \mathbb{RP}^{d-1}$. With only a little more work one may show that in fact
\[\lim_{n \to \infty} \left\|\mathcal{L}_p^n\right\|^{\frac{1}{n}} = \lim_{n \to\infty} \left(\sum_{|\iii|=n}  \|A_\iii\|^p\right)^{\frac{1}{n}},\]
and under mild algebraic non-degeneracy conditions on the matrices $A_i$, a rather longer argument shows that $\varrho_p(A_1,\ldots,A_N)$ is the largest eigenvalue of $\mathcal{L}_p$ acting on $C^\alpha(\mathbb{RP}^{d-1})$ if $\alpha>0$ is chosen sufficiently small (see for example \cite[Th\'eor\`eme 8.8]{GuLe04}). This suggests the idea of calculating $\varrho_p(A_1,\ldots,A_N)$ by approximating the operator $\mathcal{L}_p$ with a large matrix representing the action of the matrices $A_i$ on a discretised version of $\mathbb{RP}^{d-1}$. This approach was previously described in \cite[{\S}IV.A]{Va10} but does not seem to have been investigated in detail. A version of this method was also suggested in \cite[\S8]{Mo18} for the purpose of estimating the Hausdorff dimensions of some self-affine limit sets.

To give a concrete example, in the case $d=2$ write $u(\theta):=(\cos \theta,\sin \theta)$ for each $\theta \in [0,\pi)$ and for $\overline{u},\overline{v}$ let $[\overline{u},\overline{v})$ denote the shorter of the two arcs in $\mathbb{RP}^1$ from $\overline{u}$ to $\overline{v}$, including the former endpoint but not the latter. Fix an integer $n \geq 1$. For each $i=1,\ldots,N$ define an $n \times n$ matrix $B_i=[b_{jk}^{(i)}]_{j,k=0}^{n-1}$ by $b_{jk}^{(i)}:=\|A_iu(j\pi/n)\|^p$ if $\overline{Au(j\pi/n)} \in [\overline{u(k\pi/n)},\overline{u((k+1)\pi/n)})$ and $b_{jk}^{(i)}:=0$ otherwise. Define now the matrix $B:=\sum_{i=1}^N B_i$. Since $B$ corresponds to a version of $\mathcal{L}_p$ acting on functions defined on a discretisation of $\mathbb{RP}^1$ into $n$ evenly-spaced points, we expect that for large $n$ the spectral radius of $B$ should give a reasonable approximation to $\rho(\mathcal{L}_p)=\varrho_p(A_1,\ldots,A_N)$. In principle it may be possible to demonstrate this rigorously using the methods of \cite{KeLi99}, but this does not seem to have so far been attempted in the literature and is certainly a problem beyond the scope of this article.

For two-dimensional matrices this method appears to yield approximations accurate to several decimal places in a tolerable amount of time (see \S\ref{se:examples} below) and it is apparent from the definition that the effect of increasing the number of matrices $N$ has at worst a polynomial effect on the running time of the algorithm. However the size of the matrix required in order to discretise $\mathbb{RP}^{d-1}$ into a mesh of prescribed size $\varepsilon$ rises exponentially with the dimension $d$, suggesting that this method is unlikely to be very useful for matrices which are not of low dimension. The question also arises of whether better estimates may be obtained by adapting the mesh locally so as to include more mesh points in regions where the derivative of one of the maps $\overline{u} \mapsto \overline{A_iu}$ is large and fewer mesh points where it is small. Since the principal purpose of this article is to introduce the new algorithm given by Theorem \ref{th:main}, we leave these questions to other investigators.


\section{Overview of the proof of Theorem \ref{th:main}}\label{se:outline}

In the previous subsection we observed that $\varrho_p(A_1,\ldots,A_N)$ admits an interpretation as the leading eigenvalue of a linear operator on an infinite-dimensional function space and considered the possibility of approximating such an operator directly by operators on finite-dimensional spaces. This is however not the only mechanism by which the leading eigenvalue of an operator may be calculated. In order to describe our chosen alternative we will briefly and informally review some concepts from the theory of trace-class linear operators; thorough formal treatments of this topic may be found in e.g. \cite{GoGoKr00,Si79}.

 If an operator $\mathscr{L}$ on an infinite-dimensional Hilbert space has the property that the sequence of approximation numbers
\[\mathfrak{s}_n(\mathscr{L}):=\inf\left\{\|\mathscr{L}-\mathscr{F}\|\colon \mathrm{rank}\,\mathscr{F}<n\right\}\]
is summable then it is called \emph{trace-class}. If this is the case then $\mathscr{L}$ is a compact operator (since it is a limit in the norm topology of a sequence of finite-rank operators) and therefore its spectrum consists of $0$ together with a finite or infinite set of eigenvalues, each of finite algebraic multiplicity, which has no nonzero accumulation points. It is not difficult to see that $\mathfrak{s}_n(\mathscr{L}^k) \leq \|\mathscr{L}^{k-1}\|\mathfrak{s}_n(\mathscr{L})$ for every $k,n \geq 1$ by direct manipulation of the definition and consequently every power of a trace-class operator is also trace-class. If $\mathscr{L}$ is a trace-class operator on $\mathscr{H}$ with finite or infinite sequence of nonzero eigenvalues $(\lambda_n)_{n=1}^M$, it is classical that the series $\sum_{n=1}^M \lambda_n$ converges absolutely to a quantity which is called the \emph{trace} of $\mathscr{L}$ and denoted $\tr \mathscr{L}$. Moreover the quantity
\[\det(I-z\mathscr{L}):=\prod_{n=1}^M(1-z\lambda_k),\]
called the \emph{Fredholm determinant} of $\mathscr{L}$, defines an entire holomorphic function in the variable $z$ with power series $\sum_{n=0}^\infty a_nz^n$, say. It is also classical that in this case the zeros of $z \mapsto \det(I-z\mathscr{L})$ are precisely the reciprocals of the nonzero eigenvalues of $\mathscr{L}$ and that additionally
\begin{equation}\label{eq:qe}a_n= \sum_{i_1<i_2<\cdots <i_n} \lambda_{i_1}\cdots \lambda_{i_n} = \frac{1}{n!} \sum_{n_1+\cdots +n_k=n} \prod_{i=1}^k \left(-\frac{\tr \mathscr{L}^{n_i}}{n_i}\right)\end{equation}
for every $n\geq 1$, where $a_0:=1$ and where $\lambda_k$ is interpreted as zero if $k > M$. It follows that if the traces $\tr \mathscr{L}^k$ can be easily calculated for $k=1,\ldots,n$, say, then an approximation $\sum_{k=0}^n a_kz^k$ to the Fredholm determinant  can be constructed using \eqref{eq:qe} and it might be hoped that the smallest positive real root of the polynomial  $\sum_{k=0}^n a_kz^k$ would provide a good estimate for the reciprocal of the leading eigenvalue $\rho(\mathscr{L})$ of $\mathscr{L}$ as long as the remainder $\sum_{k=n+1}^\infty a_kz^k$ is extremely small. In view of the equation \eqref{eq:qe} it follows that if the sequence $(\lambda_n)_{n=1}^M$ can be shown to decay stretched-exponentially then this remainder will in fact be super-exponentially small, and this is indeed the approach which we will take in estimating $\varrho_p(A_1,\ldots,A_N)$. This general  approach to estimating dynamical quantities via operator eigenvalues has been previously exploited in a number of prior articles of which we note \cite{Ba07,JePo02,JePo05,Mo18,Po10,PoJe00,PoWe08}.

The proof of Theorem \ref{th:main} therefore proceeds via the introduction of a trace-class operator $\mathscr{L}$ on a Hilbert space $\mathscr{H}$ with the properties required by the argument sketched above: a stretched-exponential estimate on the singular values $\mathfrak{s}_n(\mathscr{L})$ (which implies a stretched-exponential estimate on the eigenvalues via Weyl's inequality), the property $\rho(\mathscr{L})=\varrho_p(A_1,\ldots,A_N)$,  and a simple, computationally-feasible formula for the sequence of traces $\tr \mathscr{L}^n$. The following result from \cite{Mo18} saves us the necessity of constructing such an operator from first principles:
\begin{theorem}[{\cite[Corollary 5.1]{Mo18}}]\label{th:opter}
Let $d,N \geq 2$, let $A_1,\ldots,A_N$ be real $d \times d$ matrices and suppose that  $(\mathcal{K}_1,\ldots,\mathcal{K}_m)$ is a multicone for $(A_1,\ldots,A_N)$ with transverse-defining vector $w \in \mathbb{R}^d$. Then there exists a nonempty bounded open subset $\Omega$ of the complex hyperplane  $\{z \in \mathbb{C}^d \colon \langle z,w\rangle=1\}$ such that the following properties hold. Let $\mathcal{A}^2(\Omega)$ denote the separable complex Hilbert space of holomorphic functions $\Omega \to \mathbb{C}$ for which the integral $\int_{\Omega} |f(z)|^2 dV(z)$ is finite, where $V$ denotes $2(d-1)$-dimensional Lebesgue measure on $\Omega$. For each $p \in \mathbb{C}$ define an operator $\mathscr{L}_p \colon \mathcal{A}^2(\Omega) \to \mathcal{A}^2(\Omega)$ by
\[\left(\mathscr{L}_pf\right)(z):=\sum_{i=1}^N \left(\frac{\langle A_i z,w\rangle}{\sign \Re(\langle A_iz,w\rangle)}\right)^p f\left(\langle A_iz,w\rangle^{-1}A_iz\right).\]
Then the operators $\mathscr{L}_p$ are well-defined bounded linear operators on $\mathcal{A}^2(\Omega)$ and:
\begin{enumerate}[(i)]
\item
There exist $C,\kappa,\gamma>0$ such that for all $p \in \mathbb{C}$ and $n \geq 1$ we have
\[\mathfrak{s}_n(\mathscr{L}_p) \leq C\exp\left(\kappa |p|-\gamma n^{\frac{1}{d-1}}\right).\]
In particular each $\mathscr{L}_p$ is trace-class.
\item
For every $p \in \mathbb{C}$ and $n \geq 1$ we have
\[\tr \mathscr{L}_p^n =\sum_{|\iii|=n} \rho(A_\iii)^p \prod_{j=2}^d \left(1-\frac{\lambda_j(A_\iii)}{\lambda_1(A_\iii)}\right)^{-1}=\sum_{|\iii|=n}\frac{\lambda_1(A_\iii)^{d-1}\rho(A_\iii)^p}{p_{A_\iii}'(\lambda_1(A_\iii))}\]
\item
For every $p \in \mathbb{R}$ the spectral radius of $\mathscr{L}_p$ is equal to
\[\lim_{n \to \infty} \left(\sum_{|\iii|=n}\left\|A_\iii\right\|^p\right)^{\frac{1}{n}}.\]
\item
For all $p \in \mathbb{R}$ the spectral radius of $\mathscr{L}_p$ is a simple eigenvalue of $\mathscr{L}_p$ and there are no other eigenvalues of the same modulus.
\end{enumerate}
\end{theorem}

Theorem \ref{th:main} can thus be seen as a version of the eigenvalue-problem approach discussed in the previous section, but one which takes advantage of the special additional structure of trace-class operators. Note that since trace-class operators are compact operators they are very far from being invertible, and indeed an important feature of the hypotheses of Theorem \ref{th:opter} is that the transformations $A_i$ map a (not necessarily connected) patch of $\mathbb{RP}^{d-1}$ strictly inside itself -- which results in a non-invertible action on the associated function space -- as opposed to acting transitively on $\mathbb{RP}^{d-1}$. This feature is precisely the content of the multicone hypothesis, and indeed the non-invertibility of the action on $\mathbb{RP}^{d-1}$ is critical in constructing a space on which the  operators $\mathscr{L}_p$ can act in a trace-class manner. As such any extension of the method of Theorem \ref{th:main} to families of matrices with non-real eigenvalues is therefore likely to be impossible since such matrices would tend to act transitively on the phase space $\mathbb{RP}^{d-1}$, preventing the construction of a suitable domain for a trace-class operator to act upon.



\section{Proof of Theorem \ref{th:main}}\label{se:proof}

The following result summarises the classical results on traces and determinants of trace-class operators on Hilbert spaces which will be required in our proof. It is a combination of several results from \cite[\S3]{Si79}, with the exception of the determinant formula for $a_n$ which may be found instead in, for example, \cite[Theorem 6.8]{Si77} or \cite[Theorem IV.5.2]{GoGoKr00}.
\begin{theorem}\label{th:dets}
Let $\mathscr{H}$ be a complex separable Hilbert space, let $\mathscr{L}$ be a trace-class operator acting on $\mathscr{H}$, and define $a_0:=1$ and 
\begin{align*}a_n:=&\frac{(-1)^n}{n!}\det\begin{pmatrix}
\tr \mathscr{L} & n -1&  0&\cdots &0 &0\\
\tr \mathscr{L}^2&\tr \mathscr{L} &n-2 &\cdots &0 &0\\
\tr \mathscr{L}^3&\tr \mathscr{L}^2&\tr \mathscr{L}  &\ddots &0 &0\\
\vdots & \vdots & \vdots & \ddots  &\ddots& \vdots\\
\tr \mathscr{L}^{n-1} &\tr \mathscr{L}^{n-2}&\tr \mathscr{L}^{n-3}&\cdots &\tr \mathscr{L}  &1\\
\tr \mathscr{L}^{n} &\tr \mathscr{L}^{n-1}&\tr \mathscr{L}^{n-2}&\cdots &\tr \mathscr{L}^2 &\tr \mathscr{L} 
\end{pmatrix}\\=&\sum_{k=1}^n \frac{(-1)^k}{k!}\sum_{\substack{n_1,\ldots,n_k \geq 1\\n_1+\cdots +n_k=n}} \prod_{\ell=1}^k \frac{t_{n_\ell}}{n_\ell}\end{align*}
for every $n \geq 1$. Then the power series $\mathscr{D}(z):=\sum_{n=0}^\infty a_nz^n$ converges for all $z \in \mathbb{C}$. The function $\mathscr{D} \colon \mathbb{C} \to \mathbb{C}$ is holomorphic, the zeros of $\mathscr{D}$ are precisely the reciprocals of the nonzero eigenvalues of $\mathscr{L}$, and the degree of each zero of $\mathscr{D}$ is equal to the algebraic multiplicity of the corresponding eigenvalue of $\mathscr{L}$. Moreover the coefficients $a_n$ satisfy the estimate
\[|a_n| \leq \sum_{i_1<i_2<\cdots<i_n} \mathfrak{s}_{i_1}(\mathscr{L}) \cdots  \mathfrak{s}_{i_n}(\mathscr{L}) \]
for every $n \geq 1$.
\end{theorem}
We also require the following elementary lemma:
\begin{lemma}\label{le:sus}
For each $\gamma,\alpha>0$ there exists a constant $K=K(\alpha,\gamma)>0$ such that
\[\sum_{n=m}^\infty \exp\left(-\gamma n^\alpha\right) \leq K\exp\left(-\frac{\gamma}{2^{1+\alpha}} m^\alpha\right)\]
for all $m\geq 1$.
\end{lemma}
\begin{proof}
Fix $\gamma$ and $\alpha$. By adjusting the constant $K$ if necessary we may without loss of generality assume $m \geq 2$. Define
\[C:=\sup\left\{u^{\frac{1}{\alpha}-1} \exp\left(-\frac{\gamma}{2}u\right)\colon u\geq 1\right\}>0.\]
Since clearly $e^{-\gamma n^\alpha}\leq \int_{n-1}^n e^{-\gamma t^\alpha}dt$ for every integer $n$ we have 
\begin{align*}\sum_{n=m}^\infty \exp\left(-\gamma n^\alpha\right) &\leq \int_{m-1}^\infty \exp\left(-\gamma t^\alpha \right)dt\\
&= \frac{1}{\alpha}\int_{\left(m-1\right)^\alpha}^\infty u^{\frac{1}{\alpha}-1} \exp\left(-\gamma u \right)du\\
&\leq \frac{C}{\alpha} \int_{\left(m-1\right)^\alpha}^\infty  \exp\left(-\frac{\gamma}{2} u \right)du\\
&= \frac{C}{\alpha} \exp\left(-\frac{\gamma}{2}\left(m-1\right)^\alpha \right) \leq\frac{C}{\alpha} \exp\left(-\frac{\gamma}{2^{1+\alpha}}m^\alpha \right)
\end{align*}
for every $m \geq 2$ and the result follows.
\end{proof}

We may now begin the proof of Theorem \ref{th:main}. Fix $A_1,\ldots,A_N$ and $p \in \mathbb{R}$ as in Theorem \ref{th:main}. By Theorem \ref{th:opter} there exist a complex separable Hilbert space $\mathscr{H}$ and a trace-class linear operator $\mathscr{L}_p \colon \mathscr{H} \to \mathscr{H}$ such that $\varrho_p(A_1,\ldots,A_N)$ is a simple isolated eigenvalue of $\mathscr{L}_p$, such that all other eigenvalues have absolute value strictly smaller than $\varrho_p(A_1,\ldots,A_N)$, such that 
\[\tr \mathscr{L}_p^n = \sum_{|\iii|=n}  \rho(A_\iii)^p \prod_{j=2}^d \left(1-\frac{\lambda_j(A_\iii)}{\lambda_1(A_\iii)}\right)=\sum_{|\iii|=n}\frac{\lambda_1(A_\iii)^{d-1}\rho(A_\iii)^p}{p_{A_\iii}'(\lambda_1(A_\iii))}\]
for every $n \geq 1$ and such that there exist constants $C_1,\gamma_1>0$ such that $\mathfrak{s}_n(\mathscr{L}_p) \leq C_1\exp( - \gamma_1 n^{\frac{1}{d-1}})$ for every $n\geq 1$. Define the sequence $(t_n)$ in accordance with Theorem \ref{th:main} and note that we have $t_n=\tr \mathscr{L}^n_p$ for every $n \geq 1$. For each $n \geq 0$ let $a_n$ be as defined in Theorem \ref{th:dets} and note that this coincides with the definition of the sequence $a_n$ in Theorem \ref{th:main}. We claim that there exist $C_2, \gamma_2>0$ such that
\begin{equation}\label{eq:est}|a_n|\leq C_2\exp\left(-\gamma_2 n^{\frac{d}{d-1}}\right)\end{equation}
for every $n \geq 1$. To see this let $n \geq 1$ and observe that by Theorem \ref{th:dets}
\begin{align*}|a_n| \leq \sum_{i_1<\cdots <i_n} \mathfrak{s}_{i_1}(\mathscr{L}_p)\cdots  \mathfrak{s}_{i_n}(\mathscr{L}_p)  &\leq \sum_{i_1<\cdots <i_n} \prod_{k=1}^n C_1\exp\left(-\gamma_1 i_k^{\frac{1}{d-1}}\right)\\
&=C_1^n \sum_{i_1<\cdots<i_n} \exp\left(-\gamma_1 \sum_{k=1}^n i_k^{\frac{1}{d-1}}\right)\\
&\leq C_1^n \sum_{i_1=1}^\infty \cdots \sum_{i_n=n}^\infty  \exp\left(-\gamma_1 \sum_{k=1}^n i_k^{\frac{1}{d-1}}\right)\\
&= C_1^n \prod_{k=1}^n \sum_{i_k=k}^\infty \exp\left(-\gamma_1 i_k^{\frac{1}{d-1}}\right)\\
&\leq C_1^nK^n \prod_{k=1}^n \exp\left(-\frac{\gamma_1}{2^{\frac{d}{d-1}}} k^{\frac{1}{d-1}}\right)\\
&= C_1^nK^n  \exp\left(-\frac{\gamma_1}{2^{\frac{d}{d-1}}} \sum_{k=1}^n k^{\frac{1}{d-1}}\right)\\
&\leq C_1^nK^n\exp\left(-\frac{(d-1)\gamma_1}{d2^{\frac{d}{d-1}}}  n^{\frac{d}{d-1}}\right)\end{align*}
where we have used Lemma \ref{le:sus} with $\alpha=\frac{1}{d-1}$ and have also used the elementary inequality
\[\sum_{k=1}^n k^{\frac{1}{d-1}} \geq \int_0^n t^{\frac{1}{d-1}} dt = \frac{d-1}{d}n^{\frac{d}{d-1}}\]
which is valid since the series is an upper Riemann sum of the integral. The claim follows easily. 

Now define a function $\mathscr{D} \colon \mathbb{C} \to \mathbb{C}$ by $\mathscr{D}(z):=\sum_{n=0}^\infty a_nz^n$. It is clear from the estimate \eqref{eq:est} that this power series has infinite radius of convergence and therefore $\mathscr{D}$ is a well-defined holomorphic function on $\mathbb{C}$. By Theorem \ref{th:dets} we have $\mathscr{D}(z)=\det(I-z\mathscr{L}_p)$ for all $z \in \mathbb{C}$ and the zeros of $\mathscr{D}$ are precisely the reciprocals of the nonzero eigenvalues of $\mathscr{L}_p$ with the degree of each zero being equal to the algebraic multiplicity of the corresponding eigenvalue. By Theorem \ref{th:opter}, $\varrho_p(A_1,\ldots,A_N)$ is the largest eigenvalue of $\mathscr{L}_p$ in absolute value and is a simple eigenvalue. It follows that we may choose a circular contour $\Gamma$ in $\mathbb{C}$ which is centred somewhere on the real line, passes through $0$, encloses $1/\varrho_p(A_1,\ldots,A_N)$ and does not enclose or intersect the reciprocal of any eigenvalue of $\mathscr{L}_p$ other than $\varrho_p(A_1,\ldots,A_N)$. Let $c \in \mathbb{R}$ and $R>0$ denote the centre point and radius of $\Gamma$ respectively. Since $\Gamma$ does not intersect the reciprocal of any eigenvalue of $\mathscr{L}_p$ the function $\mathscr{D}$ does not have any zeros on $\Gamma$, so by compactness
\[\inf_{|z-c|=R} |\mathscr{D}(z)|>0.\]
For each $n \geq 1$ define a function $\mathscr{D}_n \colon \mathbb{C} \to \mathbb{C}$ by $\mathscr{D}_n(z):=\sum_{k=0}^n a_kz^k$. Obviously each $\mathscr{D}_n$ is a polynomial and is therefore holomorphic on $\mathbb{C}$. Via Lemma \ref{le:sus} the estimate \eqref{eq:est} implies
\begin{align}\label{eq:beans}\sup_{|z-c|\leq R} \sum_{k=n}^\infty |a_kz^k| &\leq \sum_{k=n}^\infty C_2(2R)^k\exp\left(-\gamma_2 k^{\frac{d}{d-1}}\right)\\\nonumber
&\leq \sum_{k=n}^\infty C_3 \exp\left(-\gamma_3 k^{\frac{d}{d-1}}\right) \leq C_4 \exp\left(-\gamma_4 n^{\frac{d}{d-1}}\right)\end{align}
for all $n \geq 1$ and some suitable constants $C_3,C_4,\gamma_3,\gamma_4>0$. In particular
\begin{equation}\label{eq:amyfuckingmetcalfe}\lim_{n \to \infty} \sup\left\{|\mathscr{D}(z)-\mathscr{D}_n(z)|\colon |z-c|\leq R\right\}=0\end{equation}
and therefore there exists $n_0 \geq 1$ such that for all $n \geq n_0$
\[\sup_{|z-c|=R} |\mathscr{D}(z)-\mathscr{D}_n(z)|<\inf_{|z-c|=R}|\mathscr{D}(z)|.\]
Applying Rouch\'e's theorem on the circular contour $\Gamma$ we deduce that for all $n\geq n_0$ the functions $\mathscr{D}$ and $\mathscr{D}_n$ have the same number of zeros inside the contour $\Gamma$, and the total degree of the zeros inside $\Gamma$ is the same for the function $\mathscr{D}$ as it is for the function $\mathscr{D}_n$. Since $\mathscr{D}$ has a unique zero inside $\Gamma$ and that zero is simple this means that $\mathscr{D}_n$ has a unique zero inside $\Gamma$ for all large enough $n$, and this zero is simple. Call this zero $r_n$. Since $\mathscr{D}_n$ is a polynomial with real coefficients its zeros are symmetrically located with respect to reflection in the real axis. Since the contour $\Gamma$ is circular with real centre, a zero of $\mathscr{D}_n$ is enclosed by $\Gamma$ if and only if the complex conjugate of that zero is also so enclosed. It follows that the complex conjugate of $r_n$ is also enclosed by the contour $\Gamma$ and is therefore also a zero of $\mathscr{D}_n$. But $\mathscr{D}_n$ has a unique zero inside $\Gamma$. These statements can only be compatible if $r_n$ is equal to its own complex conjugate, and we conclude that $r_n$ is real. Since $r_n$ is enclosed by $\Gamma$ and is real it necessarily lies on the interval $(0,2R)$ and is the unique zero of $\mathscr{D}_n$ on that interval. In particular it is the smallest positive zero of the polynomial $\mathscr{D}_n$.

Define $r_\infty:=1/\varrho_p(A_1,\ldots,A_N) \in (0,2R)$. To complete the proof of the theorem we will show that
\[\left|\frac{1}{r_\infty}-\frac{1}{r_n}\right| = O\left(\exp\left(-\gamma_4 n^{\frac{d}{d-1}}\right)\right).\]
We first require a lower bound for the derivative $\mathscr{D}'(z)$ for $z$ close to $r_\infty$. Since $r_\infty= 1/\varrho_p(A_1,\ldots,A_N)$ is a simple zero of $\mathscr{D}$ we have $\mathscr{D}'(r_\infty) \neq 0$, and since it is also necessarily an isolated zero we may choose $\delta>0$ such that $|\mathscr{D}'(z)|\neq 0$ for all $z \in \mathbb{C}$ with $|z-r_\infty|\leq \delta$, such that $\mathscr{D}(z)\neq 0$ for all $z \in \mathbb{C}$ with $0<|z-r_\infty|\leq \delta$, and such that the closed disc of radius $\delta$ and centre $r_\infty$ is enclosed by the contour $\Gamma$. Since by compactness 
\[\inf_{|z-r_\infty|=\delta} |\mathscr{D}(z)|>0\]
it follows via \eqref{eq:amyfuckingmetcalfe} in the same manner as before that there exists $n_1 \geq n_0$ such that for all $n \geq n_1$
\[\sup_{|z-r_\infty|=\delta} |\mathscr{D}(z)-\mathscr{D}_n(z)|<\inf_{|z-r_\infty|=\delta}|\mathscr{D}(z)|.\]
Applying Rouch\'e's theorem again, this time to the circular contour with centre $r_\infty$ and radius $\delta$, we see that for each $n \geq n_1$ there is a unique zero of $\mathscr{D}_n$ within distance $\delta$ of $r_\infty$. Since the disc of radius $\delta$ and centre $r_\infty$ is enclosed by $\Gamma$, and $\Gamma$ encloses a unique zero of $\mathscr{D}_n$, we conclude that this zero must be $r_n$ and therefore $|r_n-r_\infty|<\delta$ for all $n \geq n_1$.

Now define
\[\kappa:=\inf\left\{|\mathscr{D}'(z)| \colon |z-r_\infty| \leq \delta\right\}>0.\]
Since $\mathscr{D}_n$ is a polynomial with real coefficients it takes only real values when restricted to $\mathbb{R}$ and therefore the same is true of $\mathscr{D}$ since it is the pointwise limit of $\mathscr{D}_n$ as $n \to \infty$. Let $n \geq n_1$ and suppose that $r_n \neq r_\infty$. By the Mean Value Theorem it follows that there exists a real number $t$ in the interval from $r_n$ to $r_\infty$ such that
\[\frac{\mathscr{D}(r_n)-\mathscr{D}(r_\infty)}{r_n-r_\infty} =\mathscr{D}'(t).\]
Since clearly $|r_\infty-t|\leq |r_\infty-r_n|\leq \delta$ we have $|\mathscr{D}'(t)|\geq \kappa$ and therefore
\[|r_n-r_\infty| \leq \kappa^{-1}|\mathscr{D}(r_n)-\mathscr{D}(r_\infty)|.\]
This inequality is obviously also true for integers $n \geq n_1$ such that $r_n=r_\infty$. In particular for all $n \geq n_1$ we have
\[|r_n-r_\infty| \leq \kappa^{-1}|\mathscr{D}(r_n)-\mathscr{D}_n(r_n)|\]
using the fact that $\mathscr{D}_n(r_n)=0=\mathscr{D}(r_\infty)$. Thus
\[|r_n-r_\infty| \leq \kappa^{-1}|\mathscr{D}(r_n)-\mathscr{D}_n(r_n)|=\kappa^{-1}\left|\sum_{k=n+1}^\infty a_kr_n^k\right| \leq \kappa^{-1}C_4 \exp\left(-\gamma_4 n^{\frac{d}{d-1}}\right) \]
for all $n \geq n_1$ using \eqref{eq:beans}. We in particular have $\lim_{n \to \infty} r_n=r_\infty$. If $n_2 \geq n_1$ is taken large enough that for all $n \geq n_2$ we have $r_n \geq \frac{1}{2}r_\infty$, then for all $n \geq n_2$ we have
\[\left|\frac{1}{r_n}-\frac{1}{r_\infty}\right| = \frac{|r_n-r_\infty|}{r_nr_\infty} \leq \frac{|r_n-r_\infty|}{\frac{1}{2}r_\infty^2}\leq \frac{2C_4}{\kappa r_\infty^2} \exp\left(-\gamma_4 n^{\frac{d}{d-1}}\right)\]
and this completes the proof of the theorem. 

\section{Example: a pair of matrices considered by Jungers and Protasov}\label{se:examples}
In the article \cite{JuPr11} the $p$-radius of the pair $(A_1,A_2)$ defined by
\[A_1:=\begin{pmatrix}\frac{1}{5}&0\\ \frac{1}{5}&\frac{3}{5} \end{pmatrix},\qquad A_2:=\begin{pmatrix}\frac{3}{5}&\frac{1}{5}\\ 0&\frac{1}{5} \end{pmatrix}\]
was investigated motivated by its connection with Chaikin's subdivision schemes and the $L^p$ regularity of refinable functions. The reader may easily check that if we define
\[X:=\begin{pmatrix}3&-1\\-1&3\end{pmatrix}\]
then the matrices $X^{-1}A_1X$ and $X^{-1}A_2X$ are both positive, so the pair $(A_1,A_2)$ strictly preserves a cone and Theorem \ref{th:main} may be applied thereto. The results of applying the various methods of estimation to $\varrho_{3.5}(A_1,A_2)$ are tabulated in Figures \ref{fi:gs}--\ref{fi:ghting} below. The reader will notice that by far the best results are those obtained by Theorem \ref{th:main}: the estimate obtained by evaluating all products $A_\iii$ of length up to 12 yields the estimate $0.19773 29868 07531 90957\ldots$ which is empirically accurate to all decimal places shown. Estimates of comparable complexity using the method of \S\ref{ss:jupr} give only the first two decimal places, albeit rigorously; the na{\"\i}ve upper and lower estimates described in \S\ref{ss:basic} are not even sufficient to establish the first significant digit of $\varrho_{3.5}(A_1,A_2)$. The methods of \S\ref{ss:van} and \S\ref{ss:eig} perform somewhat better, being able to give non-rigorous estimates accurate to several decimal places.  We also observe that the upper estimate arising from logarithmic convexity, 
\[\varrho_{3.5}(A_1,A_2) \simeq \sqrt{\varrho_3(A_1,A_2)\varrho_4(A_1,A_2)}=\sqrt{\rho\left(A_1^{\otimes 3}+A_2^{\otimes 3}\right)\rho\left(A_1^{\otimes 4}+A_2^{\otimes 4}\right)},\]
gives a rigorous upper bound of 
\[\varrho_{3.5}(A_1,A_2) \leq 0.19867 20360\ldots\]
which, remarkably, is more accurate than several of the other methods employed. Applying Theorem \ref{th:main} with $n=20$ gives the estimate
\[\varrho_{3.5}(A_1,A_2)\simeq 0.1977329868075319095734771033479503703640246341567\ldots\]
which is empirically accurate to all decimal places shown and provides the value of the estimate \eqref{eq:showoff} mentioned in the introduction.

\begin{figure}[H]
\begin{center}
\begin{tabular}{ccc} \toprule
    {$n$} & {Na\"{\i}ve upper estimate} & {Na\"{\i}ve lower estimate}  \\ \midrule
    1  & 0.41014 02388   &0.00003 71719 \\
    2  & 0.29717 45163   &0.00265 32644 \\
    3  & 0.26212 69438   &0.01107 32061  \\
    4  & 0.24497 10624  &0.02270 50356   \\
        5  & 0.23489 87259 &0.03497 48389 \\
6  & 0.22831 70520   &0.04666 81491 \\
7  & 0.22369 66328  &0.05735 30955  \\
8  & 0.22028 14135 &0.06694 70201\\
 9  & 0.21765 70884  &---\\
10  & 0.21557 86195  &---\\                           
    11 & 0.21389 22442   &---\\
    12 & 0.21249 67903  &---  \\ \bottomrule
\end{tabular}

\caption{The rigorous upper and lower estimates \eqref{eq:element-up} and \eqref{eq:element-down} applied to the pair $(A_1,A_2)$ with $p=3.5$. The upper estimate requires the computation of $2^n$ matrix products and the lower estimate $4^n$ products. For $n>8$ the lower estimate was omitted due to the large number of products to be computed and the poor quality of the estimates.}
\end{center}\label{fi:gs}
\end{figure}

\begin{figure}[H]
\begin{center}
\begin{tabular}{ccc} \toprule
{Mesh size} & Estimate\\ \midrule
10 & 0.22765 40788\\ 
100 & 0.19986 86395\\ 
1000 & 0.19785 78266\\ 
10000 &0.19774 13329\\ 
100000 & 0.19773 40963\\ 
\bottomrule
\end{tabular}
\caption{Estimates of $\varrho_{3.5}(A_1,A_2)$ given by the eigenvalue method described in \S\ref{ss:eig}.}
\end{center}
\end{figure}

\begin{figure}[H]
\begin{center}
\begin{tabular}{ccc} \toprule
    {Sample length} & {Number of runs} & {Resampled Monte Carlo estimate}  \\ \midrule
    10  & 10  & 0.20663 64774 \\
    10 & 100  &0.19336 14906 \\
    10  & 1000  &0.19472 39505 \\\midrule
    100  & 10  & 0.19078 48295   \\ 
    100 & 100    &0.19724 80647 \\
    100 & 1000  &0.19706 73206  \\ \midrule
        1000  & 10  &0.19171 01011   \\ 
    1000 & 100    &0.19752 20499\\
    1000 & 1000  &0.19768 32282 \\\midrule
   10000&10& 0.19460 13140\\
       10000&100& 0.19737 86045\\
           10000&1000& 0.19766 64507\\
     \bottomrule
\end{tabular}
\caption{Some representative instances of J. Vanneste's resampled Monte Carlo scheme applied to the pair $(A_1,A_2)$ with $p=3.5$ over various parameter ranges.}
\end{center}
\end{figure}

\begin{figure}
\begin{center}
\begin{tabular}{ccc} \toprule
    {$n$} & {Upper estimate $\mathsf{b}_p(n)^{\frac{1}{n}}$} & {Lower estimate $d^{-\frac{1}{n}}\mathsf{a}_p(n)^{\frac{1}{n}}$}  \\ \midrule
    1  & 0.20779 00346   &0.08095 43081 \\
    2  & 0.20474 70800  &0.14134 17665 \\
    3  & 0.20294 52224  &0.16241 04530  \\
    4  & 0.20180 54158   &0.17198 46647   \\
        5  & 0.20104 31937  &0.17732 22741 \\
6  & 0.20050 82647   &0.18073 86055 \\
7  & 0.20011 68386   &0.18313 47477  \\
8  & 0.19982 00191   &0.18492 14944 \\
 9  & 0.19958 80621  &0.18631 15004 \\
10  & 0.19940 21599   &0.18742 65582 \\                           
    11 & 0.19924 99839   &0.18834 21232 \\
    12 & 0.19912 31811  &0.18910 78446  \\ \bottomrule
\end{tabular}

\caption{Rigorous upper and lower estimates given by the algorithm of Jungers and Protasov applied to the pair $(A_1,A_2)$ with $p=3.5$.}
\end{center}
\end{figure}

\begin{figure}
\begin{center}
\begin{tabular}{cc} \toprule
    {$n$} & {Estimate $1/r_n$}  \\ \midrule
    1  & 0.50193 86416 68481 22831 92327 \\
    2  & --- \\
    3  & 0.25470 11941 19890 64296 65247 \\
    4  & ---  \\
        5  &0.19747 18486 52733 86575 36851 \\
6  & 0.19773 76208 73169 67676 89071\\
7  & 0.19773 30386 40809 03204 40047 \\
8  & 0.19773 29865 81371 43318 96314\\
 9  & 0.19773 29868 07433 20636 81181 \\
10  & 0.19773 29868 07532 62503 56803\\                           
    11 & 0.19773 29868 07531 90980 60910\\
    12 & 0.19773 29868 07531 90957 29023\\ \bottomrule
\end{tabular}

\caption{Estimates of $\varrho_{3.5}(A_1,A_2)$ provided by Theorem \ref{th:main}. For $n=2,4$ the polynomial $\sum_{k=0}^n a_kx^k$ has no real roots and the quantity $1/r_n$ is undefined.}
\end{center}\label{fi:ghting}
\end{figure}

\section{Conclusions}
We have introduced a new method for estimating the $p$-radius of low-cardinality sets of positive or dominated matrices and investigated its effectiveness in the case of a particular pair of matrices considered by Jungers and Protasov in connection with applications to Chaikin's subdivision scheme. We have compared its results to those of a number of other estimation methods in the case of that example and obtained results apparently accurate to within an absolute error of approximately $10^{-20}$, versus approximately $10^{-2}$ to $10^{-6}$ for rival methods. 

The new method has the disadvantage that the number of matrix products which must be computed in order to obtain the $n^{\mathrm{th}}$ approximation to $\varrho_p(A_1,\ldots,A_N)$ grows approximately as $N^n$. In particular if the number of matrices $N$ being considered is greater than around 4, the computational burden of producing accurate results may be prohibitively large. This disadvantage is however shared by the methods of \S\ref{ss:basic} and \S\ref{ss:jupr}. In view of this consideration, when $N$ is large the methods of \S\ref{ss:van} and \S\ref{ss:eig} may be preferable. Our method also, as presently formulated, does not provide a rigorous estimate of its own accuracy, and if rigorous bounds are sought then the method of \S\ref{ss:jupr}, possibly in combination with the logarithmic-convexity bound \eqref{eq:logconvex} may be applied instead. For two-dimensional positive matrices it seems likely that an effective bound on the error $|\varrho_p(A_1,\ldots,A_N)-1/r_n|$ could be given by adapting the arguments of \cite{JePo16,JuMo19}, but in higher dimensions this would require new technical results in order to bound the cardinality of the relative covers arising in the application of \cite[Theorem 4.7]{BaJe08} to the action of real linear maps on projective slices of complex cones.

\section{Acknowledgements}
This research was supported by the Leverhulme Trust (Research Project Grant number RPG-2016-194). \bibliographystyle{acm}
\bibliography{polecat}
\end{document}